\documentclass[11pt]{article}

\usepackage[utf8]{inputenc} 
\usepackage{amsmath}
\usepackage{graphicx}
\usepackage{subcaption}
\usepackage[top=30mm, bottom=30mm, left=27mm, right=27mm]{geometry}
\usepackage{amsfonts}
\usepackage{mathtools}
\usepackage{subcaption}
\usepackage{amssymb}
\usepackage{amsthm}
\usepackage{chngpage}
\usepackage{hyperref}
\hypersetup{
	colorlinks=true,
	linkcolor=black,
	citecolor=black,
	filecolor=black,      
	urlcolor=black,
}
\usepackage{enumerate}
\usepackage{makeidx}
\usepackage{bm}
\usepackage{thmtools}
\usepackage{thm-restate}

\makeindex

\makeatletter
\renewenvironment{proof}[1][\proofname] {\par\pushQED{\qed}\normalfont\topsep6\p@\@plus6\p@\relax\trivlist\item[\hskip\labelsep\bfseries#1\@addpunct{.}]\ignorespaces}{\popQED\endtrivlist\@endpefalse}
\makeatother

\newtheorem{proposition}{Proposition}
\newtheorem{lemma}[proposition]{Lemma}
\newtheorem{theorem}[proposition]{Theorem}

\newtheorem{question}[proposition]{Question}
\usepackage{amsthm}

\theoremstyle{definition}

\newtheorem*{remark*}{Remark}
\newtheorem*{theorem*}{Theorem}

\mathcode`l="8000
\begingroup
\makeatletter
\lccode`\~=`\l
\DeclareMathSymbol{\lsb@l}{\mathalpha}{letters}{`l}
\lowercase{\gdef~{\ifnum\the\mathgroup=\m@ne \ell \else \lsb@l \fi}}%
\endgroup

\title{Large hypergraphs without tight cycles}
\author{Barnabás Janzer\thanks{Department of Pure Mathematics and Mathematical Statistics, University of Cambridge, Wilberforce Road, Cambridge CB3 0WB, United Kingdom. Email: bkj21@cam.ac.uk. This work was supported by EPSRC DTG.}}
\date{\vspace{-21pt}}

\begin{document}
	\maketitle
	
\begin{abstract}
An $r$-uniform tight cycle of length $\ell>r$ is a hypergraph with vertices $v_1,\dots,v_\ell$ and edges $\{v_i,v_{i+1},\dots,v_{i+r-1}\}$ (for all $i$), with the indices taken modulo $\ell$. It was shown by Sudakov and Tomon that for each fixed $r\geq 3$, an $r$-uniform hypergraph on $n$ vertices which does not contain a tight cycle of any length has at most $n^{r-1+o(1)}$ hyperedges, but the best known construction (with the largest number of edges) only gives $\Omega(n^{r-1})$ edges. In this note we prove that, for each fixed $r\geq 3$, there are $r$-uniform hypergraphs with $\Omega(n^{r-1}\log n/\log\log n)$ edges which contain no tight cycles, showing that the $o(1)$ term in the exponent of the upper bound is necessary.
\end{abstract}

\section{Introduction}

A well-known basic fact about graphs states that a graph on $n$ vertices containing no cycle of any length has at most $n-1$ edges, with this upper bound being tight. To find generalisations of this result (and other results concerning cycles) for $r$-uniform hypergraphs with $r\geq 3$, we need a corresponding notion of cycles in hypergraphs. There are several types of hypergraph cycles for which Turán-type problems have been widely studied, including Berge cycles and loose cycles \cite{bollobas2008pentagons,furedi2014hypergraph,gyHori2006triangle,gyaari2012hypergraphs,jiang2018cycles,kostochka2015turan}. In this note we will consider tight cycles, for which it appears to be rather difficult to obtain extremal results.

Given positive integers $r\geq 2$ and $l>r$, an $r$-uniform tight cycle of length $l$ is a hypergraph with vertices $v_1,\dots,v_l$ and edges $\{v_i,v_{i+1},\dots,v_{i+r-1}\}$ for $i=1,\dots,l$, with the indices taken modulo $l$. Observe that for $r=2$ a tight cycle of length $l$ is just a cycle of length $l$ in the usual sense. Let $f_r(n)$ denote the maximal number of edges that an $r$-uniform hypergraph on $n$ vertices can have if it has no subgraph isomorphic to a tight cycle of any length. So $f_2(n)=n-1$. It is easy to see that the hypergraph obtained by taking all edges containing a certain point is tight-cycle-free, giving a lower bound $f_r(n)\geq \binom{n-1}{r-1}$. Sós and independently Verstra{\"e}te (see \cite{verstraete2016extremal}) raised the problem of estimating $f_r(n)$, and asked whether the lower bound $\binom{n-1}{r-1}$ is tight. This question was answered in the negative by Huang and Ma \cite{huang2019tight}, who showed that for $r\geq 3$ there exists $c_r>0$ such that if $n$ is sufficiently large then $f_r(n)\geq (1+c_r)\binom{n-1}{r-1}$. Very recently, Sudakov and Tomon \cite{sudakov2020extremal} showed that $f_r(n)\leq n^{r-1+o(1)}$ for each fixed $r$, and commented that it is widely believed that the correct order of magnitude is $\Theta(n^{r-1})$. The main result of this paper is the following theorem, which disproves this conjecture.

\begin{theorem}\label{thm_tightcycles}
	For each fixed $r\geq 3$ we have $f_r(n)=\Omega(n^{r-1}\log{n}/\log\log{n})$. In particular, $f_r(n)/n^{r-1}\to \infty$ as $n\to \infty$.
\end{theorem}

The upper bound of Sudakov and Tomon \cite{sudakov2020extremal} is $n^{r-1}e^{c_r\sqrt{\log n}}$, although they remark that it might be possible to use their approach to get an upper bound of $n^{r-1}(\log{n})^{O(1)}$.\medskip

Concerning tight cycles of a given length, we mention the following interesting problem of Conlon (see \cite{mubayi2011hypergraph}), which remains open.

\begin{question}[Conlon]
	Given $r\geq 3$, does there exist some $c=c(r)$ constant such that whenever $l>r$ and $l$ is divisible by $r$ then any $r$-uniform hypergraph on $n$ vertices which does not contain a tight cycle of length $l$ has at most $O(n^{r-1+c/l})$ edges?
\end{question}

Note that we need the assumption that $l$ is divisible by $r$, otherwise a complete $r$-uniform $r$-partite hypergraph has no tight cycle of length $l$ and has $\Theta(n^{r})$ edges.

\section{Proof of our result}

The key observation for our construction is the following lemma.

\begin{lemma}\label{lemma_maintightlemma}
	Assume that $n,k,t$ are positive integers and $G_1, \dots, G_t$ are edge-disjoint subgraphs of $K_{n,n}$ such that no $G_i$ contains a cycle of length at most $2k$. Assume furthermore that $kt\leq n$. Then there is a tight-cycle-free $3$-partite $3$-uniform hypergraph on at most $3n$ vertices having $k\sum_{i=1}^{t}|E(G_i)|$ hyperedges.
\end{lemma}

\begin{proof}
	Let the two vertex classes of $K_{n,n}$ be $X$ and $Y$, and let $Z=[t]\times [k]$. (As usual, $[m]$ denotes $\{1,\dots,m\}$.) Our 3-uniform hypergraph has vertex classes $X, Y, Z$ and hyperedges \begin{equation*}\{\{x,y,z\}:x\in X,y\in Y,z\in Z, z=(i,s)\textnormal{ for some } i\in[t]\textnormal{ and } s\in [k],\textnormal{ and } \{x,y\}\in E(G_i)\}.\end{equation*}
	In other words, for each $G_i$ we add $k$ new vertices (denoted $(i,s)$ for $s=1,\dots,k$), and we replace each edge of $G_i$ by the $k$ hyperedges obtained by adding one of the new vertices corresponding to $G_i$ to the edge.
	
	We need to show that our hypergraph contains no tight cycles. Since our hypergraph is 3-partite, it is easy to see that any tight cycle is of the form $x_1y_1z_1x_2y_2z_2\dots x_ly_lz_l$ (for some $l\geq 2$ positive integer) with $x_j\in X, y_j\in Y, z_j\in Z$ for all $j$. Assume that $z_1=(i,s_1)$. Then $\{x_1,y_1\},\{y_1,x_2\},\{x_2,y_2\}\in E(G_i)$. But $\{x_2,y_2\}\in  E(G_i)$ implies that $z_2$ must be of the form $(i,s_2)$ for some $s_2$. Repeating this argument, we deduce that there are $s_j\in [k]$ such that $z_j=(i,s_j)$ for all $j$, and $x_jy_j, y_jx_{j+1}\in E(G_i)$ for all $j$ (with the indices taken mod $l$). Hence $x_1y_1x_2y_2\dots x_ly_l$ is a cycle in $G_i$, giving $l>k$. But the vertices $z_j=(i,s_j)$ ($j=1,\dots,l$) must all be distinct, and there are $k$ possible values for the second coordinate, giving $l\leq k$. We get a contradiction, giving the result.
\end{proof}

We mention that Lemma \ref{lemma_maintightlemma} can be generalised to give $(r+r')$-uniform tight-cycle-free hypergraphs if we have edge-disjoint $r$-uniform hypergraphs $G_1,\dots, G_t$ not containing tight cycles of length at most $rk$ and edge-disjoint $r'$-uniform hypergraphs $H_1,\dots, H_t$ not containing tight cycles of length more than $r'k$. Indeed, we can take all edges $e\cup f$ with $e\in E(G_i), f\in E(H_i)$ for some $i$. (Then Lemma \ref{lemma_maintightlemma} may be viewed as the special case $r=2, r'=1$.)

\begin{lemma}\label{lemma_randomcyclefree}
	There exists $\alpha>0$ such that whenever $k\leq \alpha \log{n}/\log\log{n}$ then we can find edge-disjoint subgraphs $G_1,\dots,G_t$ of $K_{n,n}$ with $t=\lfloor n/k\rfloor$ such that no $G_i$ contains a cycle of length at most $2k$, and $\sum_{i=1}^{t}|E(G_i)|=(1-o(1))n^2$.
\end{lemma}

\begin{proof}
	It is well-known (and can be proved by a standard probabilistic argument) that there are constants $\beta, c>0$ such that if $n$ is sufficiently large and $k\leq \beta \log{n}$ then there exists a subgraph $H$ of $K_{n,n}$ which has no cycle of length at most $2k$ and has $|E(H)|\geq n^{1+c/k}$. We randomly and independently pick copies $H_1,\dots,H_t$ of $H$ in $K_{n,n}$. Let $G_1=H_1$ and $E(G_i)=E(H_i)\setminus\bigcup_{j=1}^{i-1}E(H_j)$ for $i\geq 2$. Then certainly the $G_i$ are edge-disjoint and no $G_i$ contains a cycle of length at most $2k$. Furthermore, the probability that a given edge is not contained in any $H_i$ is
	\begin{align*}
	(1-|E(H)|/n^2)^t&\leq \exp\left({-|E(H)|t/n^2}\right)\leq \exp\left({-n^{1+c/k}\lfloor n/k\rfloor /n^2}\right)=\exp\left({-n^{c/k}/k (1+o(1))}\right).
	\end{align*}
	This is $o(1)$ as long as $k\leq \alpha \log{n}/\log\log n$ for some constant $\alpha>0$.
	Therefore the expected value of $\left|\bigcup_{i=1}^t E(H_i)\right|$ is $(1-o(1))n^2$. Since $\sum_{i=1}^{t}|E(G_i)|=\left|\bigcup_{i=1}^t E(H_i)\right|$, the result follows.
\end{proof}

\begin{proof}[Proof of Theorem \ref{thm_tightcycles}]
	First consider the case $r=3$. Lemma \ref{lemma_randomcyclefree} and Lemma \ref{lemma_maintightlemma} together show that if $k\leq \alpha\log n/\log\log n$ then there is a tight-cycle-free 3-partite 3-uniform hypergraph on $3n$ vertices with $(1-o(1))kn^2$ edges. This shows $f_3(n)=\Omega(n^2\log n/\log\log n)$, as claimed.
	
	For $r\geq 4$, observe that $f_r(2n)\geq f_{r-1}(n)n$. Indeed, if $H$ is an $(r-1)$-uniform tight-cycle-free hypergraph on $n$ vertices, then we can construct a tight-cycle-free $r$-uniform hypergraph $H'$ on $2n$ vertices with $n|E(H)|$ edges as follows. The vertex set of $H'$ is the disjoint union of $[n]$ and the vertex set $V(H)$ of $H$, and the edges are $e\cup\{i\}$ with $e\in E(H)$ and $i\in [n]$. Then any tight cycle in $H'$ must be of the form $v_1v_2\dots v_{lr}$ with $v_i\in V(H)$ if $i$ is not a multiple of $r$ and $v_{i}\in [n]$ if $i$ is a multiple of $r$. But then we get a tight cycle $v_1v_2\dots v_{r-1}v_{r+1}v_{r+2}\dots v_{2r-1}v_{2r+1}\dots v_{lr-1}$ in $H$ by removing each vertex from $[n]$ from this cycle. This is a contradiction, so $H'$ contains no tight cycles. The result follows.
\end{proof}

\bibliographystyle{abbrv}
\bibliography{Bibliography}

\end{document}